\UseRawInputEncoding
\documentclass[12pt]{article}
\usepackage{amsmath}
\usepackage{amsthm}
\usepackage[colorlinks, backref=page]{hyperref}
\usepackage{hyperref}
\usepackage{cleveref}
\usepackage{amssymb}
\theoremstyle{plain}
\newtheorem{thm}{Theorem}[section]
\newtheorem{cor}{Corollary}[section]
\newtheorem{lem}{Lemma}[section]

\theoremstyle{definition}
\newtheorem{defn}{Definition}[section]
\theoremstyle{plain}

\newtheorem{rem}{Remark}[section]
\numberwithin{equation}{section}
\numberwithin{equation}{section}

\usepackage{titlesec}

\allowdisplaybreaks
\title{A remark of  Ricci-Bourguignon harmonic soliton }
\author{Xiang-Zhi Cao\thanks{School of Information Engineering, Nanjing Xiaozhuang University, Nanjing 211171, China}\,\,\thanks{caoxiangzhi@njxzc.edu.cn}}
\date{}
\begin{document}
	\maketitle
	\begin{abstract}
	In this paper, we investigate the triviality of Ricci-Bourguignon harmonic solitons. We also use the results of V-harmonic map to investigate the property of Ricci harmonic soliton.
	\end{abstract}

\section{Introduction}

   M\"{u}ller \cite{MR2961788} introduced Ricci-harmonic flow, which is defined as  follows:
for a closed manifold $M$,  given  a map $\phi$ from $M$ to some closed target manifold $N$ :
   \[
   \frac{\partial}{\partial t} g=-2 \mathrm{Rc}+2 \alpha \nabla \phi \otimes \nabla \phi,  \frac{\partial}{\partial t} \phi=\tau_{g} \phi
   \]
   where $g(t)$ is a time-dependent metric on $M, \mathrm{Rc}$ is the corresponding Ricci curvature, $\tau_{g} \phi$ is the tension field of $\phi$ with respect to $g$ and $\alpha$ is a positive constant (possibly time dependent).

 Later, people began to study its long existence such as these works( \cite{MR3406473}\cite{MR4421057}\cite{MR3782539}\cite{MR3653093}\cite{MR3547931}\cite{MR3054623}\cite{MR3729736}). Fang and Zheng \cite{MR3532234} gived the heat kernel estimates along harmonic-Ricci flow. One can refer to the works (\cite{MR3632897}\cite{MR3062555}\cite{MR3715688}\cite{MR3175257}) for the studies related to Ricci-harmonic flow.

  Azami \cite{MR4013315} introduced Ricci-Bourguignon harmonic flow, which is
  \begin{equation*}
  	\begin{aligned}
  		\begin{cases}
  			 \frac{\partial}{\partial t} g=-2 \mathrm{Rc}-2\rho Rg+2 \alpha \nabla \phi \otimes \nabla \phi,\\
  			   \frac{\partial}{\partial t} \phi=\tau_{g} \phi
  		\end{cases}
  	\end{aligned}
  \end{equation*}

Next, we give a definition,
   \begin{defn}[ Almost Ricci-harmonic solitons] Let $u:(M, g) \rightarrow(N, h)$ a smooth map (not necessarily harmonic map), where $(M, g)$ and $(N, h)$ are static Riemannian manifolds.  $ ((M,g),(N,h),V,u,\rho,\lambda) $ is called almost  Ricci-Bourguignon harmonic solitons if
   \begin{equation*}
   	\begin{aligned}
   		\begin{cases}
   				R c-\rho R g-\alpha \nabla u \otimes \nabla u-\frac{1}{2}\mathcal{L}_V g=\lambda g \\
   			\tau_g u+\langle\nabla u, V\rangle=0,
   		\end{cases}
   	\end{aligned}
   \end{equation*}
   where	$\alpha>0$ is a positive constant depending on $m$, $ \rho $ is a real  constant and $\lambda $ is a smooth function.
   
   When $ \lambda $ is a real  constant, it is called Ricci-Bourguignon harmonic soliton .	In particular, when $ V=-\nabla f $,   $ ((M,g),(N,h),f,u,\rho,\lambda) $ is called a gradient Ricci-Bourguignon harmonic soliton if it satisfies the coupled system of elliptic partial differential equations
   \begin{equation}\label{main1}
   	\begin{aligned}
   		\begin{cases}
   			R c-\rho R g-\alpha \nabla u \otimes \nabla u+\operatorname{Hess} f=\lambda g \\
   			\tau_g u-\langle\nabla u, \nabla f\rangle=0,
   		\end{cases}
   	\end{aligned}
   \end{equation}
 here $f: M \rightarrow \mathbb{R}$ be a smooth tunction and
   	where . The function $f$ is called the potential. One can refer to \cite{MR3464218}\cite{MR3996092}\cite{MR4597669}\cite{MR4579615}\cite{MR4408030} for the studies on Ricci-harmonic solitons.

It is obvious that almost  Ricci-Bourguignon harmonic solitons $ ((M,g),(N,h),V,u,\rho,\lambda) $ is almost Ricci harmonic soliton if $ \rho=0. $
   	Azami et al. \cite{MR4423983} gived the condition under which complete shrinking Ricci-harmonic Bourgainion soliton must be compact.
   \end{defn}
 The gradient Ricci-harmonic soliton is said to be shrinking, steady or expanding depending on whether $\lambda>0, \lambda=0$ or $\lambda<0$.

\begin{defn}
	Gradient Ricci-Bourguignon harmonic soliton is called trivial if the potential function $ f $ is constant.
\end{defn}
It can from \eqref{main1} that when $ u $ and $ f $ are constants, $ (M,g) $ must be Einstein manifold.

If $m$ is an integer with $0 \leq m \leq n$ and $\alpha$ is a scalar, then the $m$-th invariant of $\nabla^2 f$, is denoted by $S_m(f)$ and is defined by the condition (see, \cite{MR0334045}\cite{MR0474149},pages, 461 or \cite{soliton})
$$
\operatorname{det}\left(I+\alpha \nabla^2 f\right)=S_0(f)+\alpha S_1(f)+\cdots+\alpha^m S_m(f) .
$$

Our main motivation comes from this paper \cite{soliton}. On the one hand, we want to use their methods to explore Ricci-Bourguignon harmonic soliton. On the other hand, we know the  relation between V-harmonic map and Ricci-harmonic soliton, so we want to use V-harmonic map to study Ricci-harmonic soliton

This paper is organized as follows. In section 2, we will use the quantity $ S_2(f) $ to derive the triviality of Ricci-Bourguignon harmonic  soliton . In section 3, we will use the results of $ V $-harmonic map to derive the property of Ricci-harmonic solitons.

\section{Ricci-Bourguignon harmonic  soliton}
 In the sequel, we use the convetions and notations: $  Sc=Ric-\alpha \nabla u \otimes \nabla u-\rho R g $, its trace is $ S=(1-n\rho)R-\alpha |\nabla u |^2 $

\begin{lem}
 Let $ ((M^m,g),(N,h),f,u,\rho,\lambda) $ be a gradient almost Ricci-Bourguignon harmonic  soliton, then we have the following equations:
\begin{align}
\label{e6}	\operatorname{divSc} & =\frac{1}{2} \nabla S-\alpha \tau_g(u) \nabla u, \\
\label{e7}	\langle\nabla S, \nabla f\rangle & =2(m-1)\langle\nabla \lambda, \nabla f\rangle+2 S c(\nabla f, \nabla f)+2\rho R |\nabla f |^2+(m-2)\langle\nabla R,\nabla f\rangle, \\
\label{e1}	S c(\nabla f, \cdot) & =\frac{1}{2} \nabla S-(m-1) \nabla \lambda-2\rho \nabla f-(m-2)\nabla R, \\
\label{e2}	\nabla\left(S+|\nabla f|^2\right) & =2(m-1) \nabla \lambda+2 \lambda \nabla f +2\rho\langle\nabla f,\cdot\rangle+(m-2)\langle \nabla R,\cdot\rangle,
\end{align}
and
\begin{equation}\label{e3}
	\begin{aligned}
		\frac{1}{2} \Delta|\nabla f|^2  =&|\nabla^2f|^2-(m-2)\langle\nabla \lambda, \nabla f\rangle-S c(\nabla f, \nabla f)+\alpha |\langle\nabla u,\nabla f\rangle |^2\\ &-(m-2)\rho\langle\nabla R,\nabla f\rangle
	\end{aligned}
\end{equation}
\end{lem}
\begin{rem}
	When $ ((M^m,g),(N,h),f,u,\rho,\lambda) $ is gradient Ricci-harmonic solition, the terms about $ \rho,\lambda $ and $ (m-2)\nabla R $ disappear.
\end{rem}
\begin{proof}
We modify the proof of \cite[Proposition 3.1]{MR3997841} and \cite[Proposition 2.1]{MR4313552}.
Using the definition of  the tensor $ Sc $, we see that

	\begin{equation*}
		\begin{aligned}
			\frac{1}{2} \nabla R  =\operatorname{div} R c
			& =\operatorname{div} S c+\alpha \operatorname{div}(\nabla u \otimes \nabla u)+\rho g(\nabla R,\cdot) \\
			 &=\operatorname{div} S c+\alpha \tau_{g}(u) \nabla u+\frac{\alpha}{2} \nabla|\nabla u|^{2}+\rho g(\nabla R,\cdot),
					\end{aligned}
	\end{equation*}
It implies that
\begin{equation*}
	\begin{aligned}
			\operatorname{divSc} & =\frac{1}{2} \nabla S-\alpha \tau_g(u) \nabla u+(\frac{m}{2}-1)\rho g(\nabla R,\cdot)
	\end{aligned}
\end{equation*}

By the definition, we have 
\[
Sc_{i j}+\nabla_{i} \nabla_{j} f=\lambda g_{i j}
\]
Taking trace, we get
\begin{equation*}
	\begin{aligned}
		S+ \Delta f=m \lambda
	\end{aligned}
\end{equation*}
 Next, taking the covariant derivative me gives
\[
 \nabla_{i} S+\nabla_{i} \nabla_{j} \nabla_{j} f=m \nabla_{i} \lambda .
\]
iant derivatives and using the contracted second Bianchi identity, we

\begin{equation*}
	\begin{aligned}
		\nabla_{i} S=&-\nabla_{j} \nabla_{i} \nabla_{j} f+R_{i l} \nabla_{l} f+n \nabla_{i} \lambda\\
		=&-\nabla_{j}\left(-Sc_{i j}+\lambda g_{i j}\right)+R_{i l} \nabla_{l} f+n \nabla_{i} \lambda \\
		=&\frac{1}{2} \nabla_i S-\alpha \tau_g(u) \nabla u+(\frac{m}{2}-1)\rho g(\nabla R,\cdot)-\nabla_{i} \lambda+R_{i l} \nabla_{l} f+n \nabla_{i} \lambda
	\end{aligned}
\end{equation*}
Thus, we proved \cref{e1}
\begin{equation}\label{cf}
	\begin{aligned}
	\frac{1}{2} \nabla_{i} S=R_{i l} \nabla_{l} f+(n-1) \nabla_{i} \lambda-\alpha \tau_g(u) \nabla u+(\frac{m}{2}-1)\rho\nabla_i R,
	\end{aligned}
\end{equation}
which is equivalent to

\begin{equation*}
	\begin{aligned}
			\left( \frac{1}{2}-(n-1)\rho\right)  \nabla_{i} R-\frac{1}{2}\alpha \nabla_i|\nabla u |^2=R_{i l} \nabla_{l} f+(m-1) \nabla_{i} \lambda-\alpha \tau_g(u) \nabla u,
	\end{aligned}
\end{equation*}

In addition, noticing that $ \tau(u)=\langle\nabla u, \nabla f\rangle,$  we also have by \cref{cf}
\begin{equation*}
	\begin{aligned}
		&\langle\nabla S, \nabla f\rangle  \\
				 =&2(m-1)\langle\nabla \lambda, \nabla f\rangle+2 S c(\nabla f, \nabla f)+2\rho R |\nabla f |^2+(m-2)\rho\langle\nabla R,\nabla f\rangle,
	\end{aligned}
\end{equation*}
We proved \cref{e7}.

In addition,

$$
\begin{aligned}
	\nabla S & =2(m-1) \nabla \lambda+2 S c(\nabla f, \cdot)+2\rho\langle\nabla f,\cdot\rangle+(m-2)\langle \nabla R,\cdot\rangle \\
	& =2(m-1) \nabla \lambda+2(\lambda g-\operatorname{Hess} f) \nabla f+2\rho\langle\nabla f,\cdot\rangle+(m-2)\langle \nabla R,\cdot\rangle \\
	& =2(m-1) \nabla \lambda+2 \lambda \nabla f-2 \operatorname{Hess} f(\nabla f)+2\rho\langle\nabla f,\cdot\rangle+(m-2)\langle \nabla R,\cdot\rangle \\
	& =2(m-1) \nabla \lambda+2 \lambda \nabla f-\nabla|\nabla f|^2+2\rho\langle\nabla f,\cdot\rangle+(m-2)\langle \nabla R,\cdot\rangle,
\end{aligned}
$$
which implies
$$
\nabla\left(S+|\nabla f|^2\right)-2(m-1) \nabla \lambda-2 \lambda \nabla f=2\rho\langle\nabla f,\cdot\rangle+(m-2)\langle \nabla R,\cdot\rangle
$$
This is \cref{e2}

In the end, we prove \cref{e3}. By Bochner formula,  we have

$$
\begin{aligned}
	\frac{1}{2} \Delta|\nabla f|^2 & =|\operatorname{Hess} f|^2+\langle\nabla \Delta f, \nabla f\rangle+R c(\nabla f, \nabla f) \\
	& =|\operatorname{Hess} f|^2+m\langle\nabla \lambda, \nabla f\rangle-\langle\nabla S, \nabla f\rangle+R c(\nabla f, \nabla f) .
\end{aligned}
$$
Using the second equation of the proposition yields
$$
\begin{aligned}
	\frac{1}{2} \Delta|\nabla f|^2= & |\operatorname{Hess} f|^2+m\langle\nabla \lambda, \nabla f\rangle-2(m-1)\langle\nabla \lambda, \nabla f\rangle-2 S c(\nabla f, \nabla f) \\
	& \quad-2\rho R |\nabla f |^2-(m-2)\rho\langle\nabla R,\nabla f\rangle+R c(\nabla f, \nabla f) \\
	= & |\operatorname{Hess} f|^2-(m-2)\langle\nabla \lambda, \nabla f\rangle-S c(\nabla f, \nabla f)+\alpha \nabla u \otimes \nabla u(\nabla f, \nabla f)\\
	&-\rho R |\nabla f |^2-(m-2)\rho\langle\nabla R,\nabla f\rangle,
\end{aligned}
$$

\end{proof}

\begin{lem}[cf.{\cite{MR0474149}} ]\label{lem2.2}
 Let $(M, g)$ be a compact Riemannian manifold. Then the following holds:
	$$
	2\int_M  S_2(f)=\int_M \operatorname{Ric}(\nabla f, \nabla f) .
	$$

\end{lem}
\begin{lem}[cf.{ \cite{MR0474149})}]\label{lem2.2}
		 Let $(M, g)$ be a compact Riemannian manifold. If $f$ is a smooth function on $M$, then $S_2(f)$ can't be constant unless it vanishes.
\end{lem}
\begin{thm}\label{thm2.1}
(1)	If $ ((M^m,g),(N,h),f,u,\rho,\lambda) $  compact gradient shrinking  almost  Ricci-Bourguignon harmonic solitons  with constant $S_2(f)$, $  \rho R \leq 0 $ and
\begin{equation*}
	\begin{aligned}
		\int_M \langle\nabla \lambda, \nabla f\rangle+\int_M\rho\langle\nabla R,\nabla f\rangle \leq 0
	\end{aligned}
\end{equation*}

 then the soliton is trivial.

 (2)	If $ ((M^2,g),(N,h),f,u,\rho,\lambda) $  compact gradient shrinking  Ricci-Bourguignon harmonic solitons  with constant $S_2(f)$,
then the soliton is trivial.
\end{thm}
\begin{proof}[\bfseries	Proof of Theorem \ref{thm2.1}]
 As $S_2(f)$ is a constant, by Lemma \ref{lem2.2}, we infer that $S_2(f)$ vanishes. Therefore, we have
\begin{equation}\label{df}
	\begin{aligned}
		\int_M \operatorname{Ric}(\nabla f, \nabla f)=0.
	\end{aligned}
\end{equation}
 we obtain
	$$
	\int_M\langle\nabla f, \nabla S\rangle= 2(m-1)\int_M\langle\nabla \lambda, \nabla f\rangle+\int_M 2\rho R |\nabla f |^2+(m-2)\langle\nabla R,\nabla f\rangle.
	$$
	Now, by divergence theorem, we get
	$$
	-\int_M S \Delta f=2(m-1)\int_M\langle\nabla \lambda, \nabla f\rangle+\int_M 2\rho R |\nabla f |^2+(m-2)\langle\nabla R,\nabla f\rangle.
	$$
	Again, as $\lambda$ is a constant, we have
	$$
	\int_M m \lambda \Delta f+ m \int_M \langle\nabla \lambda, \nabla f\rangle=0.
	$$
	Thus, combining , we obtain
	$$
	\int_M(m\lambda-S) \Delta f=(m-2) \int_M \langle\nabla \lambda, \nabla f\rangle+\int_M 2\rho R |\nabla f |^2+(m-2)\langle\nabla R,\nabla f\rangle .
	$$
together with equation

	\begin{equation*}
		\begin{aligned}
			m\lambda-S=\Delta f,
		\end{aligned}
	\end{equation*}
	yields
	$$
	\int_M(\Delta f)^2=(m-2) \int_M \langle\nabla \lambda, \nabla f\rangle+\int_M 2\rho R |\nabla f |^2+\int_M (m-2)\rho\langle\nabla R,\nabla f\rangle.
	$$
	It follows that $\Delta f=0$ in the case (1) or (2) after noticing that  \cref{df} is
	\begin{equation*}
		\begin{aligned}
			\int_M R|\nabla f |^2=0.
		\end{aligned}
	\end{equation*} Since $ (M,g) $ is a compact Riemannian manifold, hence $f$ is constant. This completes the proof.
\end{proof}

Next,  we give an application of the Bochner formula  \cref{e3},

\begin{thm}\label{thm2-2}
	(1) If $ ((M^m,g),(N,h),f,u,\lambda) $ is a compact gradient   Ricci-harmonic solitons with constant $S_2(f)$ and $ \lambda $, then  this soliton is trivial.

	(2) If $ ((M^2,g),(N,h),f,u,\rho,\lambda) $ is a two dimensional compact gradient almost  Ricci-Bourguignon harmonic solitons with constant $S_2(f)$, then  this soliton is trivial.
\end{thm}
\begin{proof}

%

By  \cref{e3}, 	we obtain
	$$
	\begin{aligned}
		\frac{1}{2} \Delta|\nabla f|^2 	= & \left|\nabla^2 f\right|^2-(m-2)\langle\nabla \lambda, \nabla f\rangle-S c(\nabla f, \nabla f)+\alpha \nabla u \otimes \nabla u(\nabla f, \nabla f)\\
		&+\rho R |\nabla f |^2-(m-2)\langle\nabla R,\nabla f\rangle,\\
	\end{aligned}
	$$
Since	$ S_2(f) $ is constant, by \autoref{lem2.2} we get
	\begin{equation}\label{tg}
		\begin{aligned}
			\int_M \operatorname{Ric}(\nabla f, \nabla f) =0.
		\end{aligned}
	\end{equation}
Putting all these together  and using the divergence theorem yields
	$$
	\int_M\bigg(\left|\nabla^2 f\right|^2+(2-m)\langle\nabla \lambda, \nabla f\rangle+\rho R |\nabla f |^2-(m-2)\rho\langle\nabla R,\nabla f\rangle\bigg)=0 .
	$$
	From this, we can conclude the proof easily. In the case (1), since $ \lambda $ =constant, $ \rho=0, $ we have $ f $ is constant. In the case (2), since $ m=2$, and \cref{tg} is just
	\begin{equation*}
		\begin{aligned}
			\int_M R|\nabla f |^2=0.
		\end{aligned}
	\end{equation*}
\end{proof}

\section{Ricci harmonic  soliton}

\begin{thm}[c.f. \cite{MR2995205}, Theorem 2.]\label{thm3-1}
	 (1)  Let $(M, g)$ be a complete noncompact Riemannian manifold with
	$$
	\operatorname{Ric}_V:=\operatorname{Ric}^M-\frac{1}{2} L_V g \geq-A,
	$$
	where $A \geq 0$ is a constant, $\operatorname{Ric}^M$ is the Ricci curvature of $M$ and $L_V$ is the Lie derivative. Let $(X, h)$ be a complete Riemannian manifold with sectional curvature bounded above by a positive constant $\kappa$. Let $u: M \rightarrow X$ be a $V$-harmonic map such that $u(M) \subset B_R(p)$, where $B_R(p)$ is a regular ball in $X$, i.e., disjoint from the cut-locus of $p$ and $R<\frac{\pi}{2 \sqrt{\kappa}}$. If $V$ satisfies
	$$
	\langle V, \nabla r\rangle \leq v(r),
	$$
	for some nondecreasing function $v(\cdot)$ satisfying $\lim _{r \rightarrow+\infty} \frac{|v(r)|}{r}=0$, where $r$ denotes the distance function on $M$ from a fixed point $\tilde{p} \in M$, then $e(u)$ is bounded by a constant depending only on $A, \kappa$ and $R$. Furthermore, if $A=0$, namely,
	$$
	\operatorname{Ric}^M \geq \frac{1}{2} L_V g
	$$
	then $u$ must be a constant map.
	
	(2)  Let $M^m$ be a complete noncompact manifold with $\operatorname{Ric}_f \geq 0$ and $N$ be a complete Riemannian manifold with nonpositive sectional curvature. If $u: M \rightarrow N$ is a f-harmonic maps with finite weighted energy, then e(u) must be constant.
\end{thm}

\begin{thm}[cf. {\cite[Theorem 12]{MR3263201}}]\label{thm3-2}
 Let $\left(M^m, g\right)$ be a complete noncompact Riemannian manifold with
	$$
	\operatorname{Ric}_V:=\operatorname{Ric}^M-\frac{1}{2} L_V g \geq A,
	$$
	where $A \geq 0$ is a constant, $\operatorname{Ric}^M$ is the Ricci curvature of $M$ and $L_V$ is the Lie derivative. Let $\left(N^n, h\right)$ be a complete Riemannian manifold with sectional curvature bounded above by a
 negative constant $-\kappa^2(\kappa>0)$. Let $u: M \rightarrow N$ be a $V$-harmonic map such that $u(M) \subset B_c$, where $B_c$ is a horoball centered at $c(+\infty)$ with respect to a geodesic $c(t)$ parametrized by arc length. Suppose that $\|V\|_{L^{\infty}(M)}<+\infty$. If $A \geq \frac{\|V\|_{L \infty}^2}{m-1}$, then $u$ must be a constant map. If $A<\frac{\|V\|_{L \infty}^2}{m-1}$, then $\frac{e(u)}{(B \circ u)^2}$ is bounded by a constant depending only on $A, m, \kappa$ and $\|V\|_{L^{\infty}(M)}$.
\end{thm}
Thus, using the above Theorems, we can infer
\begin{thm}Let u be in the same situations as Theorem \ref{thm3-1},
	If $ ((M,g),(N,h),V,u,\lambda) $ is a non-compact Ricci-harmonic  soliton with constant $S_2(f)$ and $ \lambda $ , then it  is non-shrinking Ricci soliton, i.e, $ \lambda\leq 0 $.
\end{thm}

\begin{proof}
	By the assumption and Theroem \ref{thm3-1}, we know that $ u $ is constant. So, $ \alpha $-Ricci-harmonic  soliton is reduced to Ricci soliton. By contradiction, we know that   shrinking Ricci soliton  is Einstein manifold (cf. \cite[Theorem 1.2]{soliton}) . This is a contradiction.
\end{proof}
Similarly, by the similar argument as above  , we get
\begin{thm}Let u be in the same situations as  Theorem \ref{thm3-2},
	if $ ((M,g),(N,h),V,u,\lambda) $ is a compact  $ \alpha $-Ricci-harmonic  soliton with constant $S_2(f)$ and $ \lambda \geq \frac{\|V\|_{L \infty}^2}{m-1} $, then it   is non-shrinking Ricci soliton, i.e, $ \lambda\leq 0 $.
\end{thm}
Recall that 
\begin{thm}[cf. \cite{MR3573121}]
Let $\left( (M^m, g), (N^n, h), f, u,\lambda\right)$ be a shrinking or steady gradient Ricci-Harmonic soliton satisfying  with $\alpha>0$. If in addition the sectional curvature Sect ${ }^N$ of $N$ satisfies $K=\sup _N \operatorname{Sect}^N<\frac{\alpha}{m}$, then $u$ is a constant map.
\end{thm}
Combining with Theorem \ref{thm2-2} gives that
\begin{cor}
	Let $\left( (M^m, g), (N^n, h), f, u,\lambda\right)$ be a  compact shrinking or steady  gradient Ricci-Harmonic soliton  with $\alpha>0$, constant $S_2(f)$ and $ \lambda $. If in addition the sectional curvature Sect ${ }^N$ of $N$ satisfies $K=\sup _N \operatorname{Sect}^N<\frac{\alpha}{m}$, then $ (M,g) $ must be Einstein manifold.
\end{cor}
\bibliographystyle{acm}
\bibliography{E:/myonlybib/myonlymathscinetbibfrom2023,E:/myonlybib/low-quality-bib-to-publish}
\end{document}